\documentclass[11pt]{amsart}
\usepackage{amsmath,amsthm,amsfonts,amssymb,latexsym,epsfig}

\newtheorem{theorem}{Theorem}[section]

\newtheorem{lemma}{Lemma}[section]

\numberwithin{equation}{section}

\theoremstyle{definition}

\theoremstyle{remark}

\begin{document}
\title{A Note On Mixed Mean Inequalities}
\author{Peng Gao}
\address{Department of Computer and Mathematical Sciences,
University of Toronto at Scarborough, 1265 Military Trail, Toronto
Ontario, Canada M1C 1A4} \email{penggao@utsc.utoronto.ca}
\date{September 17, 2007}

\begin{abstract}We give a simpler proof of a result of Holland concerning a mixed arithmetic-geometric mean inequality. We also prove a result of mixed mean inequality involving the symmetric means.
\end{abstract}

\maketitle
\section{Introduction}
\label{sec1}
   Let $M_{n,r}({\bf x})$ be the generalized weighted power means:
   $M_{n,r}({\bf q}, {\bf x})=(\sum_{i=1}^{n}q_ix_i^r)^{\frac {1}{r}}$,
   where ${\bf q}=(q_1,q_2,\cdots,
   q_n)$, ${\bf x}=(x_1,x_2,\cdots,
   x_n)$, $q_i >0, 1 \leq i \leq n$ with $\sum_{i=1}^nq_i=1$. Here
   $M_{n,0}({\bf q}, {\bf x})$ denotes the limit of $M_{n,r}({\bf q}, {\bf x})$ as
   $r\rightarrow 0^{+}$. Unless specified, we always assume $x_i>0, 1 \leq i \leq n$.
    When there is no risk of confusion,
    we shall write $M_{n,r}$ for $M_{n,r}({\bf q}, {\bf x})$ and
    we also define $A_n=M_{n,1}, G_n=M_{n,0}, H_n=M_{n,-1}$.

   The celebrated
   Hardy's inequality (\cite[Theorem 326]{HLP}) asserts that for 
   $p>1, a_n \geq 0$,
\begin{equation*}
\sum^{\infty}_{n=1}\Big(\frac {\sum^n_{k=1}a_k}{n} \Big )^p \leq
\Big(\frac {p}{p-1} \Big )^p \sum^\infty_{n=1}a^p_n.
\end{equation*}

    Among the many different proofs of Hardy's inequality as well as its generalizations and extensions in the literature, one novel approach is via the mixed mean inequalities (see, for example, \cite[Theorem 7]{C&P}). By mixed mean inequalities, we shall mean the following inequalities: 
\begin{equation}
\label{8.19}
   \left ( \sum^m_{n=1}a_{m,n} \left ( \sum^m_{k=1}b_{n,k}x_k \right )^p \right
    )^{\frac 1{p}} \leq  \sum^m_{n=1}b_{m,n} \left ( \sum^m_{k=1}a_{n,k}x^p_k \right
    )^{\frac 1{p}},
\end{equation}
   where $(a_{i,j}), (b_{i,j})$ are two $m \times m$ matrices with non-negative
   entries and the above inequality are meant to hold for any vector ${\bf x} \in {\mathbb R}^{m}$ with non-negative entries. Here $p \geq 1$ and when $0< p \leq 1$ we
   want the inequality above to be reversed.

  The meaning of mixed mean becomes more clear when $(a_{i,j}), (b_{i,j})$ are weighted mean matrices. Here we say a matrix $A=(a_{n,k})$ is a weighted
    mean matrix if $a_{n,k}=0$ for $n<k$ and
\begin{equation}
\label{021}
    a_{n,k}=w_k/W_n,  ~~ 1 \leq k \leq
    n; \hspace{0.1in} W_n=\sum^n_{i=1}w_i, w_i \geq 0, w_1>0.
\end{equation}

   Now we focus our attention to the case of \eqref{8.19} for $(a_{i,j})=(b_{i,j})$ being weighted mean matrices given in \eqref{021}. In this case, for fixed ${\bf x}=(x_1,\cdots, x_n),{\bf w}=(w_1, \cdots, w_n)$, we define ${\bf x}_i=(x_1, \cdots, x_i), {\bf w}_i=(w_1,
\cdots, w_i), W_i=\sum^i_{j=1}w_j$, $M_{i,
   r}=M_{i,r}({\bf x}_i)=M_{i,r}({\bf w}_i/W_i, {\bf x}_i), {\bf M}_{i,r}=(M_{1,r}, \cdots,
M_{i,r})$. Then we have the following mixed mean inequalities of Nanjundiah \cite{N} (see also \cite{B}):
\begin{theorem}
\label{thm1.0}
  Let $r >s$ and $n \geq 2$. If for $
   2 \leq k \leq n-1$, $W_nw_k-W_kw_n>0$. Then
\begin{equation*}
   M_{n,s}({\bf M}_{n,r}) \geq M_{n,r}(
   {\bf M}_{n,s}),
\end{equation*}
   with equality holding if and only if $x_1 = \cdots =x_n$.
\end{theorem}

  A very elegant proof of Theorem \ref{thm1.0} for the case $r=1, s=0$ is given by Kedlaya in \cite{Ked1}. In fact, the following Popoviciu-type inequalities were established in \cite{Ked1} (see also \cite[Theorem 9]{B}): 
\begin{theorem}
\label{thm2}
  Let $n \geq 2$. If for $2 \leq k \leq n-1$, $W_nw_k-W_kw_n>0$, then
\begin{equation*}
   W_{n-1}\Big(\ln M_{n-1,0}({\bf M}_{n-1,1})-\ln M_{n-1,1}({\bf M}_{n-1,0}) \Big) \leq W_n\Big (\ln M_{n,0}({\bf M}_{n,1})-\ln M_{n,1}({\bf M}_{n,0}) \Big)
\end{equation*}
   with equality holding if and only if $x_n=M_{n-1,0}=M_{n-1,1}({\bf M}_{n-1,0})$.
\end{theorem}
  It is easy to see that the case $r=1, s=0$ of Theorem \ref{thm1.0} follows from Theorem \ref{thm2}. As was pointed out by Kedlaya that the method used in \cite{Ked1} can be applied to establish both Popoviciu-type and Rado-type inequalities for mixed means for a general pair $r>s$. The details were worked out in \cite{T&T} and the following Rado-type inequalities were established in \cite{T&T}:
\begin{theorem}
\label{thm1'}
  Let $1 >s $ and $n \geq 2$. If for $
   2 \leq k \leq n-1$, $W_nw_k-W_kw_n>0$, then
\begin{equation*}
   W_{n-1}\Big (M_{n-1,s}({\bf M}_{n-1,1})-M_{n-1,1}({\bf M}_{n-1,s}) \Big ) \leq W_n \Big (M_{n,s}({\bf M}_{n,1})-M_{n,1}({\bf M}_{n,s}) \Big )
\end{equation*}
   with equality holding if and only if $x_1 = \cdots =x_n$ and the above
   inequality reverses when $s>1$.
\end{theorem}

  
   A different proof of Theorem \ref{thm1.0} for the case $r=1, s=0$ was given in \cite{Ked} and Bennett used essentially the same approach  in \cite{Ben1} and \cite{Ben2} to study \eqref{8.19} for the cases $(a_{i,j}), (b_{i,j})$ being lower triangular matrices, namely, $a_{i,j}=b_{i,j}=0$ if
$j>i$. Among other things, he showed \cite{Ben1} that inequalities \eqref{8.19} hold when $(a_{i,j}), (b_{i,j})$ are Hausdorff matrices. 

  Recently, Holland \cite{H} further improved the condition in Theorem \ref{thm1'} for the case $s=0$ by proving the following:
\begin{theorem}
\label{thm3}
  Let $n \geq 2$. If for $
   2 \leq k \leq n-1$, $W^2_k \geq w_{k+1}\sum^{k-1}_{i=1}W_i$, then
\begin{equation}
\label{1.3}
   W_{n-1}\Big (M_{n-1,0}({\bf M}_{n-1,1})-M_{n-1,1}({\bf M}_{n-1,0}) \Big ) \leq W_n \Big (M_{n,0}({\bf M}_{n,1})-M_{n,1}({\bf M}_{n,0}) \Big )
\end{equation}
   with equality holding if and only if $x_1 = \cdots =x_n$.
\end{theorem}
  
  It is our goal in this note to first give a simpler proof of the above result by modifying Holland's own approach. This is done in the next section and in Section \ref{sec 3}, we will prove a result of mixed mean inequality involving the symmetric means.

\section{A Proof of Theorem \ref{thm3}}
\label{sec 2} \setcounter{equation}{0}
   First, we recast \eqref{1.3} as
\begin{equation}
\label{2.3}
   G_{n}({\bf A}_{n})-\frac {W_{n-1}}{W_n}G_{n-1}({\bf A}_{n-1})-\frac {w_n}{W_n}G_n \geq 0.
\end{equation}
  We now note that
\begin{eqnarray}
  G_{n}({\bf A}_{n}) &=& \Big(G_{n-1}({\bf A}_{n-1}) \Big )^{W_{n-1}/W_n}A^{w_{n}/W_n}_n, \nonumber \\
\label{1.6}
  G_{n-1}({\bf A}_{n-1}) &=& A_n\prod^{n-1}_{i=1}\Big(\frac {A_i}{A_{i+1}}\Big )^{W_i/W_{n-1}}.
\end{eqnarray}
  We may assume that $x_k > 0, 1 \leq k \leq n$ and the case $x_k=0$ for some $k$ will follow by continuity. Thus on dividing $G_{n}({\bf A}_{n})$ on both sides of \eqref{2.3} and using \eqref{1.6}, we can recast \eqref{2.3} as:
\begin{equation}
\label{1.7}
   \frac {W_{n-1}}{W_n}\prod^{n-1}_{i=1}\Big(\frac {A_i}{A_{i+1}}\Big )^{W_iw_n/(W_{n-1}W_n)}+\frac {w_n}{W_n}\prod^{n}_{i=1}\Big(\frac {x_i}{A_{i}}\Big )^{w_i/W_n} \leq 1.
\end{equation}
   We now express $x_i=(W_iA_i-W_{i-1}A_{i-1})/w_i, 1\leq i \leq n$ with $W_0=A_0=0$ to recast \eqref{1.7} as
\begin{equation*}
   \frac {W_{n-1}}{W_n}\prod^{n-1}_{i=1}\Big(\frac {A_i}{A_{i+1}}\Big )^{W_iw_n/(W_{n-1}W_n)}+\frac {w_n}{W_n}\prod^{n}_{i=1}\Big(\frac {W_iA_i-W_{i-1}A_{i-1}}{w_iA_{i}}\Big )^{w_i/W_n} \leq 1.
\end{equation*}
  We now set $y_i=A_{i}/A_{i+1}$, $1 \leq i \leq 2$ to further recast the above inequality as
\begin{equation}
\label{1.8}
   \frac {W_{n-1}}{W_n}\prod^{n-1}_{i=1}y^{W_iw_n/(W_{n-1}W_n)}_i+\frac {w_n}{W_n}\prod^{n-1}_{i=1}\Big(\frac {W_{i+1}}{w_{i+1}}-\frac {W_{i}}{w_{i+1}}y_i\Big )^{w_{i+1}/W_n} \leq 1.
\end{equation}
   It now follows from the assumption of Theorem \ref{thm3} that
\begin{equation*}
   c_n=1-\sum^{n-1}_{i=1}\frac {W_iw_n}{W_{n-1}W_n} \geq 0,
\end{equation*}
   so that by the arithmetic-geometric mean inequality we have
\begin{equation}
\label{1.9}
   \prod^{n-1}_{i=1}y^{W_iw_n/(W_{n-1}W_n)}_i=1^{c_n} \prod^{n-1}_{i=1}y^{W_iw_n/(W_{n-1}W_n)}_i\leq \sum^{n-1}_{i=1}\frac {W_iw_ny_i}{W_{n-1}W_n}+1-\sum^{n-1}_{i=1}\frac {W_iw_n}{W_{n-1}W_n}.
\end{equation}
   Similarly, we have
\begin{equation}
\label{1.10}
   \prod^{n-1}_{i=1}\Big(\frac {W_{i+1}}{w_{i+1}}-\frac {W_{i}}{w_{i+1}}y_i\Big )^{w_{i+1}/W_n} \leq \sum^{n-1}_{i=1}\frac {w_{i+1}}{W_n}\Big(\frac {W_{i+1}}{w_{i+1}}-\frac {W_{i}}{w_{i+1}}y_i\Big )+\frac {w_1}{W_n}.
\end{equation}
   Now it is easy to see that inequality \eqref{1.8} follows on adding inequalities \eqref{1.9} and \eqref{1.10} and this completes the proof of Thorem \ref{thm3}.

\section{A Discussion on Symmetric means}
\label{sec 3} \setcounter{equation}{0}
  Let $0 \leq r \leq n $,  we recall that the $r$-th symmetric function
$E_{n,r}$ of ${\bf x}$ and its mean $P_{n,r}$ is defined by
\begin{equation*}
    E_{n,r}({\bf x})=\sum_{1 \leq i_1< \cdots <i_r \leq
    n}\prod_{j=1}^rx_{i_j}, P^r_{n,r}({\bf x})=\frac {E_{n,r}({\bf x})}{\binom{n}{r}}, 1\leq r \leq n; E_{n,0}=P_{n,0}=1.
\end{equation*}
   It is well-known that for fixed ${\bf x}$ of dimension $n$, $P_{n,r}$ is a non-increasing function of $r$ for $1 \leq r \leq n$ with $P_{n,1}=A_n, P_{n,n}=G_n$ (with weights $w_i=1$, $1 \leq i \leq n$). In view of the mixed mean inequalities for the generalized weighted power means (Theorem \ref{thm1.0}), it is natural to ask whether similar results hold for the symmetric means. Of course one may have to adjust the notion of such mixed means in order for this to make sense for all $n$. For example, when $r=3$, $n=2$, the notion of $P_{2,3}$ is not even defined. From now on we will only focus on the extreme cases of the symmetric means, namely $r=2$ or $r=n-1$. In these cases it is then natural to define $P_{1,2}=x_1$ and on recasting $P_{n,n-1}=G^{n/(n-1)}_n/H^{1/(n-1)}_n$, we see that it also natural for us to define $P_{1,0}=x_1$ (note that this is not consistent with our definition of $P_{n,0}$ above).

  We now prove a mixed mean inequality involving $P_{n,2}$ and $A_n$. We first note the following result of Marcus and Lopes \cite{M&L} (see also pp. 33-35 in \cite{B&B}):
\begin{theorem}
\label{thm8}
   Let $0  < r \leq n$ and $x_i, y_i >0$ for $i=1,2, \cdots, n$,
   then
\begin{equation*}
\label{8.160}
    P_{n,r}({\bf x}+{\bf y}) \geq P_{n,r}({\bf x})+P_{n,r}({\bf
    y}),
\end{equation*}
   with equality holding if and only if $r=1$ or there exists a
   constant $\lambda$ such that ${\bf x}=\lambda {\bf y}$.
\end{theorem}
   We also need the following Lemma of C. Tarnavas and D. Tarnavas \cite{T&T}.
\begin{lemma}
\label{lem8.2}
   Let $f: R^1 \rightarrow R^1$ be a convex function and suppose for $n \geq
   2, 1 \leq k \leq n-1$, $W_nw_k-W_kw_n>0$. Then
\begin{equation*}
\label{8.7}
   \frac 1{W_{n-1}}\sum^{n-1}_{k=1}w_kf(W_{n-1}A_k) \geq \frac
   1{W_n}\sum^{n}_{k=1}w_kf(W_nA_k-w_nx_k).
\end{equation*}
   The equality holds if and only if $n=2$ or $x_1=\cdots=x_n$ when $f(x)$ is strictly
   convex. When $f(x)$ is concave, then the
   above inequality is reversed.
\end{lemma}

   We now apply Lemma \ref{lem8.2} to obtain
\begin{lemma}
\label{lem8.3}
   For $n \geq 2$ and $w_i=1$, $1 \leq i \leq n$,
\begin{equation*}
   P_{n-1,2}((n-1){\bf A}_{n-1})  \leq  P_{n,2}(n{\bf A}_n-{\bf
   x}_n).
\end{equation*}
with equality holding in both cases if and only if $n=2$ or $x_1= \cdots
=x_n$.
\end{lemma}
\begin{proof}
   The case $n=2$ yields an identity so we may assume $n \geq 3$ here. Write $a_i=(n-1)A_i, 1 \leq i \leq n-1; b_j=nA_j-x_j, 1 \leq j
   \leq n$. Note $n\sum^{n-1}_{i=1}a_i=(n-1)\sum^n_{i=1}b_i$ and now
   Lemma \ref{lem8.2} with $f(x)=x^2$ implies
   $(n-1)\sum^n_{i=1}b^2_i \leq n\sum^{n-1}_{i=1}a^2_i$.
   On expanding
   $(n\sum^{n-1}_{i=1}a_i)^2=((n-1)\sum^n_{i=1}b_i)^2$, we obtain
\begin{eqnarray*}
\label{8.11}
   n^2\sum^{n-1}_{i=1}a^2_i+2n^2\sum_{1 \leq i \neq j \leq
   n-1}a_ia_j &=& (n-1)^2\sum^n_{i=1}b^2_i+2(n-1)^2\sum_{1 \leq i \neq j \leq
   n}b_ib_j \\
   & \leq & n(n-1)\sum^{n-1}_{i=1}a^2_i+2(n-1)^2\sum_{1 \leq i \neq j \leq
   n}b_ib_j.
\end{eqnarray*}
   Hence
\begin{equation}
\label{8.12}
   n\sum^{n-1}_{i=1}a^2_i+2n^2\sum_{1 \leq i \neq j \leq
   n-1}a_ia_j \leq 2(n-1)^2\sum_{1 \leq i \neq j \leq
   n}b_ib_j.
\end{equation}
   Using $M_{n,2} \geq A_n=P_{n,1} \geq P_{n,2}$, we obtain
\begin{equation*}
\label{8.13}
   \frac 1{n-1}\sum^{n-1}_{i=1}a^2_i \geq
   \frac 1{\binom{n-1}{2}}\sum_{1\leq i \neq j \leq n-1}a_ia_j.
\end{equation*}
   So by \eqref{8.12},
\begin{equation*}
\label{8.14}
   \frac 1{\binom{n-1}{2}}\sum_{1\leq i \neq j \leq n-1}a_ia_j \leq \frac 1{\binom{n}{2}}\sum_{1\leq i \neq j \leq
   n}b_ib_j,
\end{equation*}
   which is just what we want.
\end{proof}

   We now prove the following mixed mean inequality involving the symmetric
   means:
\begin{theorem}
\label{thm8.2}
   Let $n \geq 1$ and define ${\bf P}_{n,2}=(P_{1,2}, \cdots,
P_{n,2})$, then
\begin{equation}
\label{8.10'}
    (n-1) \Big (P_{n-1,2}(
   {\bf P}_{n-1,1})-P_{n-1,1}({\bf P}_{n-1,2}) \Big )  \leq  n \Big (P_{n,2}({\bf P}_{n,1})-P_{n,1}({\bf P}_{n,2}) \Big ),
\end{equation}
  with equality holding if and only if $x_1 = \cdots =x_n$. It
  follows that
\begin{equation*}
\label{8.10}
   P_{n,1}({\bf P}_{n,2}) \leq P_{n,2}({\bf P}_{n,1}),
\end{equation*}
   with equality holding if and only if $x_1 = \cdots =x_n$.
\end{theorem}
\begin{proof}
    It suffices to prove \eqref{8.10'} here.  We may assume $n \geq 2$ here and we shall use the idea in \cite{T&T}. Lemma \ref{lem8.3} implies that
\begin{eqnarray*}
\label{8.15}
   P_{n,2}+ (n-1) P_{n-1,2}(
   {\bf P}_{n-1,1}) & \leq & P_{n,2}+ P_{n,2}(n{\bf A}_n-{\bf x}_n) \\
   &\leq&   P_{n,2}(n{\bf A}_n-{\bf x}_n+{\bf x}_n)
   = nP_{n,2}({\bf P}_{n,1}),
\end{eqnarray*}
   where the last inequality follows from Theorem \ref{thm8} for the case $r=2$. It is easy to see that the above inequality is equivalent to \eqref{8.10} and this completes the proof.
\end{proof}

    Now we let $n \geq 1$ and define ${\bf P}_{n, n-1}=(P_{1,0}, \cdots,
P_{n,n-1})$ with $P_{1,0}=x_1$ here. Then it is interesting to see whether the following inequality holds or not:
\begin{equation*}
   P_{n,1}({\bf P}_{n, n-1}) \leq P_{n, n-1}({\bf P}_{n,1}).
\end{equation*}
   We note here that if the above inequality holds, then it is easy to deduce from it via the approach in \cite{C&P} the following Hardy-type inequality:
\begin{equation*}
   \sum^n_{i=1} \frac {G^{i/(i-1)}_i}{H^{1/(i-1)}_i}({\bf x}_i) \leq e \sum^{n}_{i=1}x_i,
\end{equation*}
  where we define $G^{1/0}_1/H^{1/0}_1=x_1$. We now end this paper by proving the following result:
\begin{theorem}
  Let $n \geq 1$ and ${\bf x} \geq {\bf 0}$. Then
\begin{equation*}
   \sum^n_{i=1} \frac {G^{i/(i-1)}_i}{H^{1/(i-1)}_i}({\bf x}_i) \leq 3 \sum^{n}_{i=1}x_i,
\end{equation*}  
   where we define $G^{1/0}_1/H^{1/0}_1=x_1$.
\end{theorem}
\begin{proof}
  We follow an approach of Knopp \cite{K} here (see also \cite{D&M}). For $i \geq 1$, we define
\begin{equation*}
  a_i=\sum^{i}_{k=1}\frac {k x_k}{i(i+1)}.
\end{equation*}
   It is easy to check by partial summation that
\begin{equation*}
  \sum^n_{i=1}a_i \leq  \sum^n_{i=1}x_i.
\end{equation*}
   Certainly we have $a_1 = x_1/2 = P_{1, 0}({\bf x}_1)/2$ and for $ i \geq 2$, we apply the inequality $P_{i, 1} \geq P_{i, i-1}$ to the numbers
$x_1/(i+1), 2x_2/(i+1), \ldots, ix_i/(i+1)$ to see that
\begin{equation*}
   a_i \geq \Big ( \frac {(i-1)!}{(i+1)^{i-1}} \Big )^{1/(i-1)}P_{i,i-1}({\bf x}_i) :=\gamma_{i}P_{i,i-1}({\bf x}_i).
\end{equation*}
  We now show by induction that $\gamma_{i} \geq 1/3$ for $i \geq 2$, equivalently, this is
\begin{equation}
\label{3.3}
  3^{i-1}(i-1)! \geq (i+1)^{i-1}.
\end{equation}
 Note first that the above inequality holds when $i=2,3$ and suppose now it holds for some $i=k \geq 3$, then by induction
\begin{equation*}
  3^{k}k! \geq 3k(k+1)^{k-1}.
\end{equation*}
   Now use $(1+1/n)^n <e$, we have
\begin{equation*}
   \frac {3k(k+1)^{k-1}}{(k+2)^{k}}= \frac {3k(k+2)}{(k+1)^{2}}\Big ( \frac {k+1}{k+2}\Big )^{k+1} \geq \frac {3k(k+2)}{e(k+1)^{2}}.
\end{equation*}
   It is easy to see that the last expression above is no less than $1$ when $k \geq 3$ and this proves 
inequality \eqref{3.3} for the case $i=k+1$ and this completes the proof of the theorem.
\end{proof}



\end{document}